\documentclass[11pt]{article}
\usepackage[left=0.5in,right=0.5in,bottom=0.75in,top=0.75in]{geometry}
\usepackage{amsmath}
\usepackage{amssymb}
\usepackage{graphicx}
\usepackage{natbib}
\usepackage{url}
\usepackage[shortlabels]{enumitem}
\numberwithin{equation}{section}
\usepackage{multicol}
\usepackage{caption}
\usepackage{ctable}
\setupctable{pos=ht!}

\providecommand{\I}{\mathbf{I}}

\renewcommand{\r}{\mathbf{r}}

\providecommand{\x}{\mathbf{x}}
\providecommand{\X}{\mathbf{X}}
\providecommand{\y}{\mathbf{y}}
\providecommand{\z}{\mathbf{z}}

\providecommand{\zero}{\mathbf{0}}

\providecommand{\lam}{\lambda}

\providecommand{\veps}{\varepsilon}

\providecommand{\bb}{\boldsymbol{\beta}}
\providecommand{\bh}{\widehat{\beta}}
\providecommand{\bbh}{\widehat{\boldsymbol{\beta}}}

\providecommand{\bvep}{\boldsymbol{\varepsilon}}

\providecommand{\sh}{\hat{\sigma}}

\providecommand{\tb}{\widetilde{\beta}}
\providecommand{\tbb}{\widetilde{\boldsymbol{\beta}}}
\providecommand{\bbj}{\boldsymbol{\beta}_{-j}}

\providecommand{\bbhj}{\widehat{\boldsymbol{\beta}}_{-j}}

\providecommand{\rr}{\tilde{\mathbf{r}}}

\providecommand{\Xj}{\mathbf{X}_{-j}}

\providecommand{\yy}{\tilde{\mathbf{y}}}

\renewcommand{\Pr}{\mathbb{P}}

\providecommand{\Ex}{\mathbb{E}}

\providecommand{\var}{\textrm{Var}}

\providecommand{\cor}{\textrm{Cor}}

\providecommand{\Norm}{\textrm{N}}

\providecommand{\IID}{\overset{\text{iid}}{\sim}}

\providecommand{\inD}{\overset{\text{d}}{\longrightarrow}}

\providecommand{\FDR}{\textrm{FDR}}

\providecommand{\ind}{\perp\!\!\!\perp}

\providecommand{\abs}[1]{\left\lvert#1\right\rvert}
\providecommand{\norm}[1]{\left\lVert#1\right\rVert}

\providecommand{\al}[2]{\begin{align}\label{#1}#2\end{align}}
\providecommand{\as}[1]{\begin{align*}#1\end{align*}}
\providecommand{\als}[2]{\begin{align}\label{#1}\begin{split}#2\end{split}\end{align}}

\renewcommand{\abstract}[1]{
 \centerline{
 \begin{minipage}{0.7\linewidth}
 \hrule
 \vskip 0.1in
  \begin{center}
    {\bf Abstract}
  \end{center}
  #1
 \vskip 0.1in
 \hrule
 \end{minipage}}
 \vskip 0.3in} 

\usepackage{titling}
\pretitle{\begin{center} \LARGE \bf}
\posttitle{\par\end{center}\vskip 0.5em}

\usepackage{algorithm}
\usepackage{algpseudocode}
\algdef{SxnE}[FOR]{For}{EndFor}[1]{\algorithmicfor\ #1 }
\usepackage{framed}
\usepackage{xcolor}
\definecolor{shadecolor}{rgb}{.9,.9,.9}
\makeatletter
\newenvironment{kframe}{%
 \def\at@end@of@kframe{}%
 \ifinner\ifhmode%
  \def\at@end@of@kframe{\end{minipage}}%
  \begin{minipage}{\columnwidth}%
 \fi\fi%
 \def\FrameCommand##1{\hskip\@totalleftmargin \hskip-\fboxsep
 \colorbox{shadecolor}{##1}\hskip-\fboxsep
     \hskip-\linewidth \hskip-\@totalleftmargin \hskip\columnwidth}%
 \MakeFramed {\advance\hsize-\width
   \@totalleftmargin\z@ \linewidth\hsize
   \@setminipage}}%
 {\par\unskip\endMakeFramed%
 \at@end@of@kframe}
\makeatother

\usepackage{fancyvrb}
\newenvironment{code}
 {\VerbatimEnvironment \begin{small}\begin{kframe}\begin{Verbatim}}
 {\end{Verbatim}\end{kframe}\end{small}}

\DefineVerbatimEnvironment{codeInput}{code}{}

\usepackage{upquote}

\graphicspath{{}}

\usepackage{amsthm}
\newtheorem{theorem}{Theorem}

\providecommand{\cA}{\mathcal{A}}

\providecommand{\cM}{\mathcal{M}}
\providecommand{\cN}{\mathcal{N}}
\providecommand{\cS}{\mathcal{S}}
\providecommand{\FD}{\textrm{FD}}
\providecommand{\mFDR}{\textrm{mFDR}}

\title{Marginal false discovery rates for penalized regression models}
\author{Patrick Breheny\\Department of Biostatistics\\University of Iowa}
\date{\today}

\begin{document}

\maketitle

\abstract{Penalized regression methods are an attractive tool for high-dimensional data analysis, but their widespread adoption has been hampered by the difficulty of applying inferential tools. In particular, the question ``How reliable is the selection of those features?'' has proved difficult to address. In part, this difficulty arises from defining false discoveries in the classical, fully conditional sense, which is possible in low dimensions but does not scale well to high-dimensional settings. Here, we consider the analysis of marginal false discovery rates for penalized regression methods. Restricting attention to the marginal FDR permits straightforward estimation of the number of selections that would likely have occurred by chance alone, and therefore provides a useful summary of selection reliability. Theoretical analysis and simulation studies demonstrate that this approach is quite accurate when the correlation among predictors is mild, and only slightly conservative when the correlation is stronger. Finally, the practical utility of the proposed method and its considerable advantages over other approaches are illustrated using gene expression data from The Cancer Genome Atlas and GWAS data from the Myocardial Applied Genomics Network.}

\section{Introduction}

Penalized regression is an attractive methodology for dealing with high-dimensional data where classical likelihood approaches to modeling break down.  However, its widespread adoption has been hindered by a lack of inferential tools.  In particular, penalized regression is very useful for variable selection, but the question ``How reliable are those selections?'' has proved difficult to address.  As I will argue in this paper, this difficulty is partially due to the complexity of defining a ``false discovery'' in the penalized regression setting.

In this paper, I will focus mainly on the lasso for linear regression models \citep{Tibshirani1996}, although the idea is very general and can be extended to a variety of other regression models and penalty functions.  In particular, suppose we use the lasso to select 10 variables from a pool of potentially important features.  This paper addresses the question, ``How many of those selections would likely have occurred by chance alone?''

There has been a fair amount of recent work on the idea of hypothesis testing in the high-dimensional penalized regression setting.  A comprehensive review is beyond the scope of this paper, but it is worth introducing two approaches to which I will compare the method proposed in this paper.  One approach is to split the sample into two parts, using the first part for variable selection and the second part for hypothesis testing.  This idea was first introduced in \citet{Wasserman2009}, who studied the problem using a single split of the data set.  \citet{Meinshausen2009a} extended this approach by considering multiple random splits and combining the results.  \citet{Dezeure2015} provides a comprehensive review of this approach and the details involved, along with procedures for limiting the overall false discovery rate through this form of testing.

An alternative approach is to test the significance of adding a variable along the solution path as the degree of penalization is relaxed, conditional on the other variables already included in the model.  Several tests fall into this category, including the covariance test \citep{Lockhart2014}, the spacing test, and an exact test \citep[][which also introduces the spacing test]{Tibshirani2016}.  The details of the tests differ, but all involve modifying classical tests for the significance of an added variable by conditioning on the fact that the added variable was not prespecified, but rather selected from a pool of potential variables.  \citet{GSell2016} provide an important extension to this work by deriving a rule, forwardStop, for selecting the stopping point along this sequence of sequential tests so as to preserve a specified false discovery rate.

False discoveries are straightforward to define in single-variable hypothesis testing: a false discovery is one that is independent of the outcome.
In regression models, however, this idea is complicated by the consideration of various kinds of conditional independence.
Typically, in regression a feature $X_j$ is considered a false discovery if it is independent of the outcome $Y$ given all other features; symbolically, if $X_j \ind Y | \{X_k\}_{k \neq j}$.
This is the perspective adopted by most work in this area, including sample splitting; here, we will refer to this as the {\em fully conditional} perspective.
The pathwise approaches use a different definition, which we will call the {\em pathwise conditional} perspective.
Letting $\cM(\lam)$ denote the set of variables with nonzero coefficients in the model at the point in the path where feature $j$ is selected, in these approaches a feature $j$ is considered a false discovery if $X_j \ind Y | \{X_k:k \in \cM_{j-}\}$.
Here, $\lam$ indexes the the various models along the lasso solution path; see \eqref{eq:lasso} for an exact definition.

This paper introduces a weaker definition of false discovery than those considered in the existing literature.
Rather than the conditional definitions given in the previous paragraph, here we consider a marginal perspective in which a selected feature $j$ is false if it is marginally independent of the outcome, the definition used in single-feature testing: $X_j \ind Y$.
Adopting a simpler definition makes it possible to estimate the expected number of false discoveries as well as their rate, which we call the {\em marginal false discovery rate}, or mFDR.
Our goal here is not to argue that the marginal FDR is always superior to either of the conditional FDRs; clearly, there are times where a conditional FDR is an important quantity of interest.
However, the mFDR is also an interesting and useful summary of feature selection accuracy.
This is especially true in high dimensions where, as we will see, attempting to control a conditional false discovery rate is often hopeless.

\section{Marginal false discovery rates}
\label{Sec:fir}

Consider the linear model with normally distributed errors:
\as{\y &= \X\bb + \bvep\\
  \veps_i &\IID \Norm(0, \sigma^2),}
where $\y$ denotes the response and $\X$ the $n \times p$ design matrix, with $n$ denoting the number of independent observations and $p$ the number of features.
Without loss of generality, we assume throughout that the response and covariates are centered so that the intercept term can be ignored, and that the features are standardized so that $\tfrac{1}{n}\sum_i x_{ij}^2 = 1$ for all $j$.
We are interested in studying the lasso estimator $\bbh$, defined as the quantity that minimizes
\al{eq:lasso}{\frac{1}{2n}\norm{\y-\X\bb}_2^2 + \lam\norm{\bb}_1,}
where $\norm{\bb}_1 = \sum_j|\beta_j|$.  Here, the least-squares loss is used to measure the fit of the model and the $L_1$ norm is used to penalize large values of $\bb$, with $\lam$ controlling the tradeoff between the two.

An appealing aspect of using the lasso to estimate $\bb$ is that the resulting estimates are sparse: some coefficients will be nonzero, but for many, $\bh_j=0$.  In this paper, we say that a feature for which $\bh_j \neq 0$ has been ``selected.''  Note that $\bbh$ changes with $\lam$ although we will suppress this in the notation; or a large enough value of $\lam$, $\bh_j=0$ for all $j$, but as we lower $\lam$, more variables will be included.  To address the expected number of features included in a lasso model by chance alone, we begin by considering the orthonormal case, then turn our attention to the general case.

\subsection{Orthonormal case}
\label{Sec:ortho}

For a given value of the regularization parameter $\lam$, let $\r=\y - \X\bbh$ denote the residuals.  The lasso \eqref{eq:lasso} is a convex optimization problem, so the Karush-Kuhn-Tucker (KKT) conditions are both necessary and sufficient for any solution $\bbh$:
\as{\frac{1}{n}\x_j'\r = \lam\,\textrm{sign}(\bh_j) \hspace{1cm} &\textrm{for all }\bh_j \neq 0\\
\frac{1}{n}\abs{\x_j'\r} \leq \lam \hspace{2.4cm} &\textrm{for all }\bh_j = 0.}

Letting $\Xj$ and $\bbj$ denote the portions of the design matrix and coefficient vector that remain after removing the $j$th feature, let $\r_j = \y-\Xj\bbhj$ denote the partial residuals with respect to feature $j$.  The KKT conditions thus imply that
\als{eq:kkt}{
  \frac{1}{n}\abs{\x_j'\r_j} &  >  \lam \qquad \textrm{for all }\bh_j \neq 0\\
  \frac{1}{n}\abs{\x_j'\r_j} &\leq \lam \qquad \textrm{for all }\bh_j = 0}
and therefore the probability that variable $j$ is selected is
\as{\Pr(\bh_j \neq 0) = \Pr\left(\frac{1}{n}\abs{\x_j'\r_j} > \lam \right).}

This indicates that if we are able to characterize the distribution of $\tfrac{1}{n}\x_j'\r_j$ under the null, we can estimate the number of false discoveries in the model.  Indeed, this is straightforward in the case of orthonormal design ($\frac{1}{n}\X'\X=\I$):
\als{eq:partial}{\frac{1}{n}\x_j'\r_j &= \frac{1}{n}\x_j'(\y-\X\bb + \X_{-j}\bb_{-j} + \x_j\beta_j - \X_{-j}\bbh_{-j})\\
  &=\frac{1}{n}\x_j'\bvep + \beta_j\\
  &\sim N(\beta_j, \sigma^2/n) .}
Thus, if $\beta_j=0$, we have
\as{\Pr\left(\frac{1}{n}\abs{\x_j'\r_j} > \lam \right) &= 2\Phi(-\lam\sqrt{n}/\sigma).}

These results are related to the expected number of false discoveries in the following theorem, the proof of which follows directly from the above by summing $\Pr(\bh_j \neq 0)$ over the set of null variables.

\begin{theorem}
  \label{Thm:ortho}
  Suppose $\frac{1}{n}\X'\X=\I$.  Then for any value of $\lam$,
  \as{\Ex\abs{\cS \cap \cN} =
    2\abs{\cN}\Phi(-\lam\sqrt{n}/\sigma),}
  where $\cS=\{j:\bh_j \neq 0\}$ is the set of selected variables and
  $\cN=\{j:\beta_j=0\}$ is the set of null variables.
\end{theorem}

To use this as an estimate, the unknown quantities $\abs{\cN}$ and $\sigma^2$ must be estimated.  First, $\abs{\cN}$ can be replaced by $p$, using the total number of variables as an upper bound for the null variables.  The variance $\sigma^2$ can be estimated by
\as{\sh^2 = \frac{\r'\r}{n-\abs{\cS}}.}
Dividing the residual sum of squares by the degrees of freedom of the lasso \citep{Zou2007} is a simple approach for estimating the residual variance, but other possibilities exist \citep[e.g.,][]{Fan2012}.
This implies the following estimate for the expected number of false discoveries:
\al{eq:fd}{\widehat{\FD} = 2p\Phi(-\sqrt{n}\lam/\sh)}
and, as an estimate of the false discovery rate:
\al{eq:fdr}{\widehat{\FDR} = \frac{\widehat{\FD}}{\abs{\cS}}.}

This straightforward estimate of the false discovery rate is made possible by the orthonormality condition assumed in \eqref{eq:partial}.  When the features that comprise $\X$ are correlated, however, the distribution of $\tfrac{1}{n}\x_j'\r_j$ is considerably more complex, as is the proper definition of what is meant by a ``false discovery.''  Nevertheless, as the rest of this paper will show, with a suitable definition of false discovery, estimator \eqref{eq:fdr} can still be used to estimate false discovery rates for variable selection in non-orthonormal settings.

\subsection{Definition in the general case}
\label{Sec:defn-gen}

 Consider the causal diagram presented in Figure~\ref{Fig:diagram}.  In this situation, variable $A$ could never be considered a false discovery: it has a direct causal relationship with the outcome $Y$.  Likewise, if variable $N$ were selected, this would obviously count as a false discovery -- $N$ has no relationship, direct or indirect, to the outcome.

\begin{figure}[htb]
  \centering
  \includegraphics[width=0.2\linewidth]{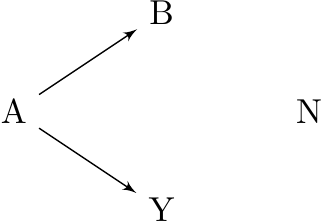}
  \caption{\label{Fig:diagram} Causal diagram depicting three types of features and their relationship to the outcome.}
\end{figure}

Variable $B$, however, occupies a gray area as far as false discoveries are concerned.
From a marginal perspective, $B$ would not be considered a false discovery, since it is not independent of $Y$.
However, $B$ and $Y$ are conditionally independent given $A$, so from a fully conditional regression perspective, $B$ is a false discovery.
Lastly, from a pathwise conditional perspective, $B$ may or may not be a false discovery, depending on whether $A$ is already included in the model or not at the point at which $B$ enters.

The central argument of this paper is that estimating the number of false selections arising from variables like $B$ is inherently complicated and requires complex approaches (either in the mathematical or computational sense); {\em however}, simple approaches like the one derived in Section~\ref{Sec:ortho} may still be used to estimate the number of false selections arising from variables like $N$.

Here, I define a {\em noise feature} to be a variable like $N$, that has no causal path (direct or indirect) between it and the outcome, and the {\em marginal false discovery rate} as the proportion of selected features that are noise variables.  Again, this definition is consistent with how false discoveries are defined in univariate testing, but differs from the existing literature in regression modeling.

It is worth addressing a technical point here: it is possible for two variables to be conditionally dependent despite being marginally independent.  Such a relationship would appear in a directed acyclic graph as the following:

\begin{center}
\includegraphics[width=0.2\linewidth]{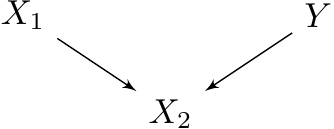}
\end{center}

\noindent As the figure indicates, such a relationship would require a feature (or features) to be caused by the outcome $Y$.
This is contrary to the usual framework in which regression models are applied, where $Y$ is (or is assumed to be) a consequence of the features used for prediction (e.g., a genetic variant can lead to heart disease, but heart disease will not change an individual's genetic sequence).
Thus, although we are actually making a stronger assumption than mere marginal independence here, in the typical regression setting the two are equivalent.
At any rate, although the term may be slightly imprecise from a technical perspective, we feel that the name ``marginal false discovery'' is easily understood and best distinguishes this quantity from false discoveries under the various conditional perspectives.

The proposed definition has several advantages.
First, when two variables (like $A$ and $B$ in Figure~\ref{Fig:diagram}) are correlated, it is often difficult to distinguish between which of them is driving changes in Y and which is merely correlated with Y.  As we will see, this causes methods using a conditional perspective to be conservative, especially in high dimensions.

Second, in many scientific applications, discovering variables like B in Figure~\ref{Fig:diagram} is not problematic.
For example, two genetic variants in close proximity to each other on a chromosome will be highly correlated.  Although it is obviously desirable to identify which of the two is the causal variant, locating a nearby variant is also an important scientific achievement, as it narrows the search to a small region of the genome for future follow-up studies.

The final advantage is clarity of interpretation.  Whether or not a feature would quality as a marginal false discovery depends only on the relationship between it and the outcome, not on whether another feature has been included in the model.
In contrast, interpreting the results from pathway-based tests such as those in \citet{Lockhart2014} and \citet{Tibshirani2016} can be challenging, as the definition of the null hypothesis is constantly changing with $\lam$.

\section{Independent noise features}

The orthonormal conditions of Section~\ref{Sec:ortho} clearly do not hold in most settings for which penalized regression is typically used.  Thankfully, they can be relaxed in two important ways that make the results more widely applicable.  First, the predictors do not have to be strictly orthogonal in order for the estimator to work; they can simply be uncorrelated.  Second, this condition of being uncorrelated applies only to the noise features -- i.e., the variables like $N$ in Figure~\ref{Fig:diagram}; variables like $A$ and $B$ can have any correlation structure.  These statements are justified theoretically in Section~\ref{Sec:ind-theory} and via simulation in \ref{Sec:ind-sim}.

\subsection{Theory}
\label{Sec:ind-theory}

To make these statements concrete, let $\cA, \cN$ partition $\{1,2,\ldots,p\}$ such that $\beta_j=0$ for all $j \in \cN$ and the following condition holds:
\as{\lim_{n \to \infty}\frac{1}{n}\X'\X = \left[ \begin{array}{cc}
      \Sigma_\cA & \zero\\
      \zero & \Sigma_\cN\end{array} \right].}
Under this definition, the opening remarks of this section can be stated precisely in the following theorem.

\begin{theorem}
  \label{Thm:main}
  Suppose $\Sigma_\cN=\I$.  Then for any $j \in \cN$ and for $\lam_n$ such that the sequence $\sqrt{n}\lam_n$ is bounded,
  \as{\frac{1}{\sqrt{n}}\x_j'\r_j \inD N(0, \sigma^2).}
\end{theorem}

Theorem~\ref{Thm:main} shows that if the noise features are uncorrelated, then in the limit $\tfrac{1}{n}\x_j'\r_j$ behaves just as it did in Section~\ref{Sec:ortho} (equation \ref{eq:partial}).
Thus, estimators \eqref{eq:fd} and \eqref{eq:fdr} still provide approximate estimators for the marginal false discovery rate in this setting, at an accuracy that improves with the sample size.
The technical condition requiring $\sqrt{n}\lam_n$ to be bounded is only necessary so that the estimate $\bbh$ will be $\sqrt{n}$-consistent.  Without it (i.e., for large values of $\lam$), $\tfrac{1}{n}\x_j'\r_j$ will converge to a random variable with a variance larger than $\sigma^2$ due to underfitting.  Note that Theorem~\ref{Thm:main} treats $p$ as fixed; extending this result to allow $p \to \infty$ would be of interest as future work, but lies outside the scope of this paper.

\subsection{Simulation}
\label{Sec:ind-sim}

To illustrate the consequences of Theorem~\ref{Thm:main}, let us carry out the following simulation study, with both a ``low-dimensional'' ($n > p$) and ``high-dimensional'' ($n < p$) component.  Motivated by Figure~\ref{Fig:diagram}, three types of features will be included:
\begin{itemize}
\item Causative: Six variables with $\beta_j =1$
\item Correlated: Each causative feature is correlated ($\rho=0.5$) with $m$ other features; $m=2$ for the low-dimensional case and $9$ for the high-dimensional case
\item Noise: Independent noise features are added to bring the total number of variables up to 60 in the low-dimensional case and 600 in the high-dimensional case
\end{itemize}
The causative, correlated, and noise features correspond to variables A, B, and N, respectively, in Figure~\ref{Fig:diagram}.  In each setting, the sample size was $n=100$, while the total number of causative/correlated/noise features was 6/12/42 for the low-dimensional setting and 6/54/540 for the high-dimensional setting.  The results of the simulation are shown in Figure~\ref{Fig:sim-ind}.

\begin{figure}[ht!]
  \centering
  \includegraphics[width=0.75\linewidth]{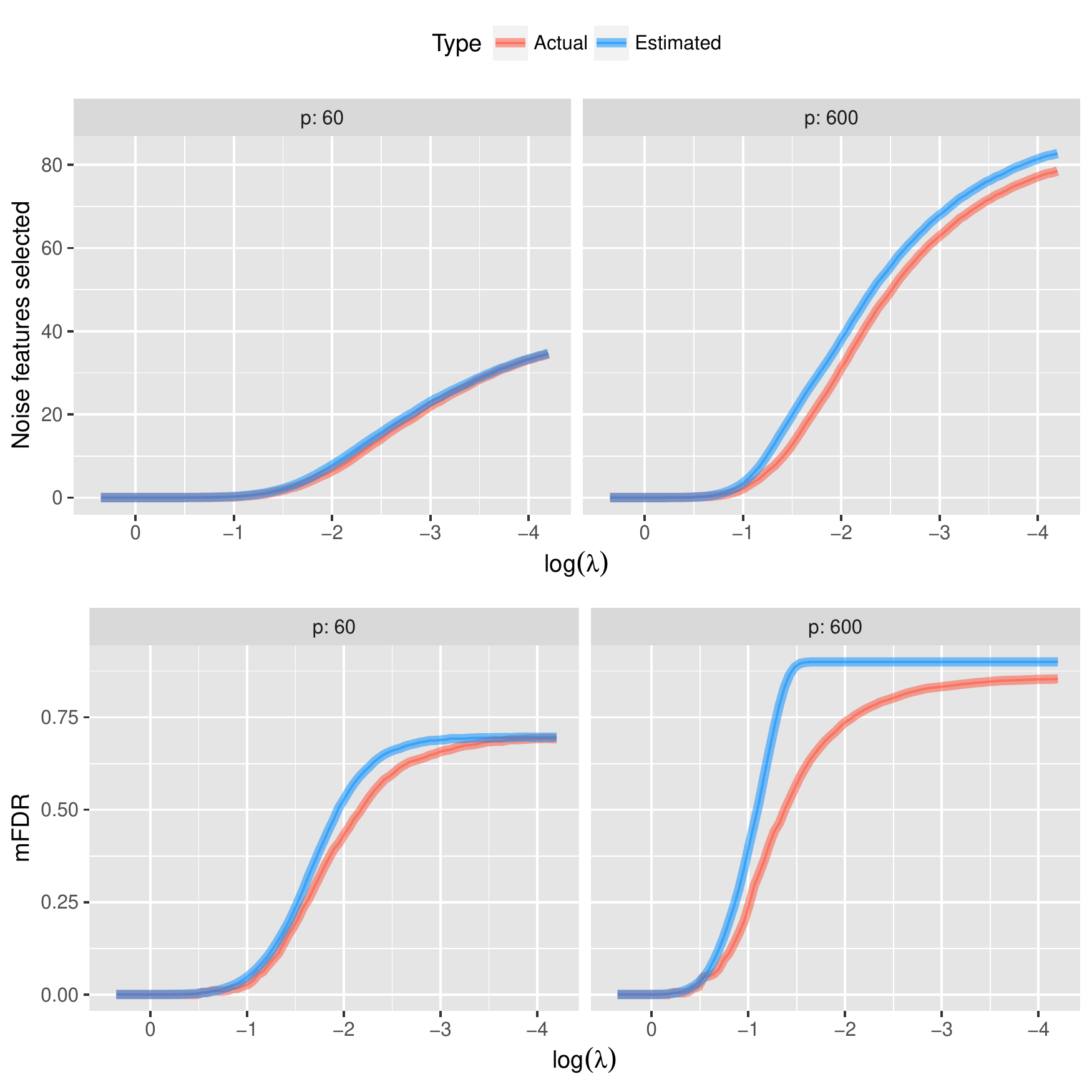}
  \caption{\label{Fig:sim-ind} Accuracy of estimators \eqref{eq:fd} and \eqref{eq:fdr} in the case of independent noise features.}
\end{figure}

As Theorem~\ref{Thm:main} implies, estimators \eqref{eq:fd} and \eqref{eq:fdr} are quite accurate, on average, when the noise features are independent.  The estimated number of noise features selected and the marginal false discovery rate (mFDR) are both somewhat conservative, as we would expect from using $p$ as an upper bound for the number of noise features (e.g., in the high-dimensional case, $p=600$ but $\abs{\cN}=540$).  However, the effect is slight in this setting.  For example, in the high-dimensional case at $\lam=0.55$, the actual inclusion rate for noise variables was 5\%, while the estimated rate was 6.5\%.

Being able to estimate the mFDR means we can use it to select the regularization parameter $\lam$.  For example, we could choose $\lam$ to be the smallest value of $\lam$ such that $\widehat{\mFDR}(\lam) < 0.1$.  Figure~\ref{Fig:sim-power} compares this approach with several other methods for selecting $\lam$ in terms of the number of each type of feature the method selects on average.  For Lasso (mFDR), univariate testing (i.e., marginal regression), sample splitting \citep[using the {\tt hdi} package,][]{Dezeure2015}, and the exact conditional path-based test \citep[using {\tt larInf} and {\tt forwardStop} from the {\tt selectiveInference} package,][]{Tibshirani2016,GSell2016}, the nominal false discovery rates were set to 10\%.  For cross-validation, the value of $\lam$ minimizing the cross-validation error was selected.

\begin{figure}
  \centering
  \includegraphics[width=0.75\linewidth]{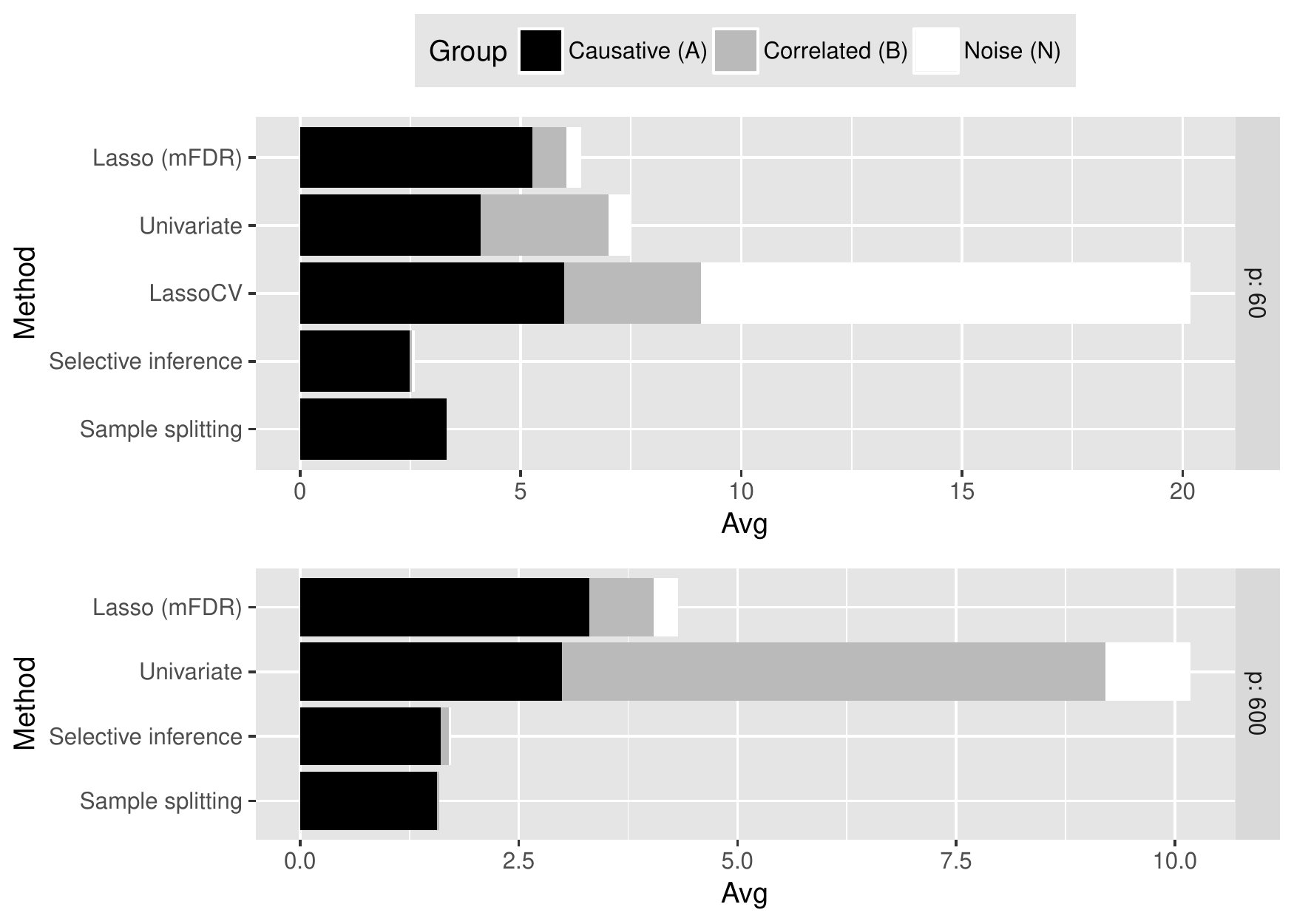}
  \caption{\label{Fig:sim-power} Average number of each type of feature selected by various methods for the simulation setup in Section~\ref{Sec:ind-sim}.  Cross-validation is omitted from the high-dimensional ($p=600$) plot due to the very large number of noise features it selects, which dominates the plot when included.}
\end{figure}

It is worth noting that both Lasso (mFDR) and univariate testing limit the fraction of selections due to noise features (to 5\% and 7\%, respectively, in the $p=60$ simulation) to the nominal rate, but claim nothing about the fraction arising from correlated features.
This is obvious from the figure for univariate testing; for Lasso (mFDR), the fraction of selected features coming from either the Correlated or Noise groups (i.e., all features with $\beta_j=0$) was 17\% in the low-dimensional setting and 23\% in the high-dimensional setting.
Nevertheless, compared to univariate testing, the Lasso (mFDR) approach has two distinct advantages.
Figure~\ref{Fig:sim-power} shows that using a penalized regression approach both diminishes the number of merely correlated features selected and improves power to detect the truly causative features.

Among the penalized regression approaches, cross-validation is a considerable outlier, with no protection against the selection of large numbers of noise features.  In the high-dimensional case, cross-validation selected over 30 noise features on average, and is omitted from the plot so as not to obscure the performance of the other methods.

The sample splitting and selective inference approaches, on the other hand, greatly limit the number of both noise features and correlated features selected by the lasso model.  The more restrictive definition of a false discovery used by these approaches, however, causes them to behave quite conservatively, and have considerably less power to detect the causative features than any of the other methods considered.

\section{Correlated noise features}
\label{Sec:cor}

The simulation results of Section~\ref{Sec:ind-sim} are something of a best case scenario for the proposed method, since the variables in $\cN$ were independent and Theorem~\ref{Thm:main} establishes that the estimator is consistent in this case.  As we will see, when noise features are correlated, estimator \ref{eq:fdr} becomes conservative.  The case of correlated noise features is significantly less mathematically tractable, making it harder to construct a rigorous proof of this claim, but here we offer some insight into why this is true and investigate the performance of the proposed method via simulation.

To illustrate why the proposed mFDR estimator becomes conservative when noise features are correlated, let us consider the $p=2$ case.
Letting $z_j = \tfrac{1}{n}\x_j'\y$, we begin by noting that $\z$ follows a multivariate normal distribution with mean $\zero$ and
\as{\var(\z) = \frac{\sigma^2}{n}\left[ \begin{array}{cc} 1 & \rho \\
      \rho & 1 \end{array} \right],}
where $\rho$ denotes the correlation between $\x_1$ and $\x_2$, and
\al{eq:conj-right-side}{2\abs{\cN}\Phi(-\lam\sqrt{n}/\sigma) = \Pr(\abs{z_1}>\lam) + \Pr(\abs{z_2}>\lam).}
A visual representation of this quantity is given in Figure~\ref{Fig:conjecture}.

The shading represents whether 0 (white), 1 (light gray), or 2 (dark gray) features will be selected under this orthogonal approximation; the quantity in \eqref{eq:conj-right-side} can be found by integrating these constant regions with respect to the joint density of $\z$.  

\begin{figure}[ht!]
  \centering
  \includegraphics[width=0.7\linewidth]{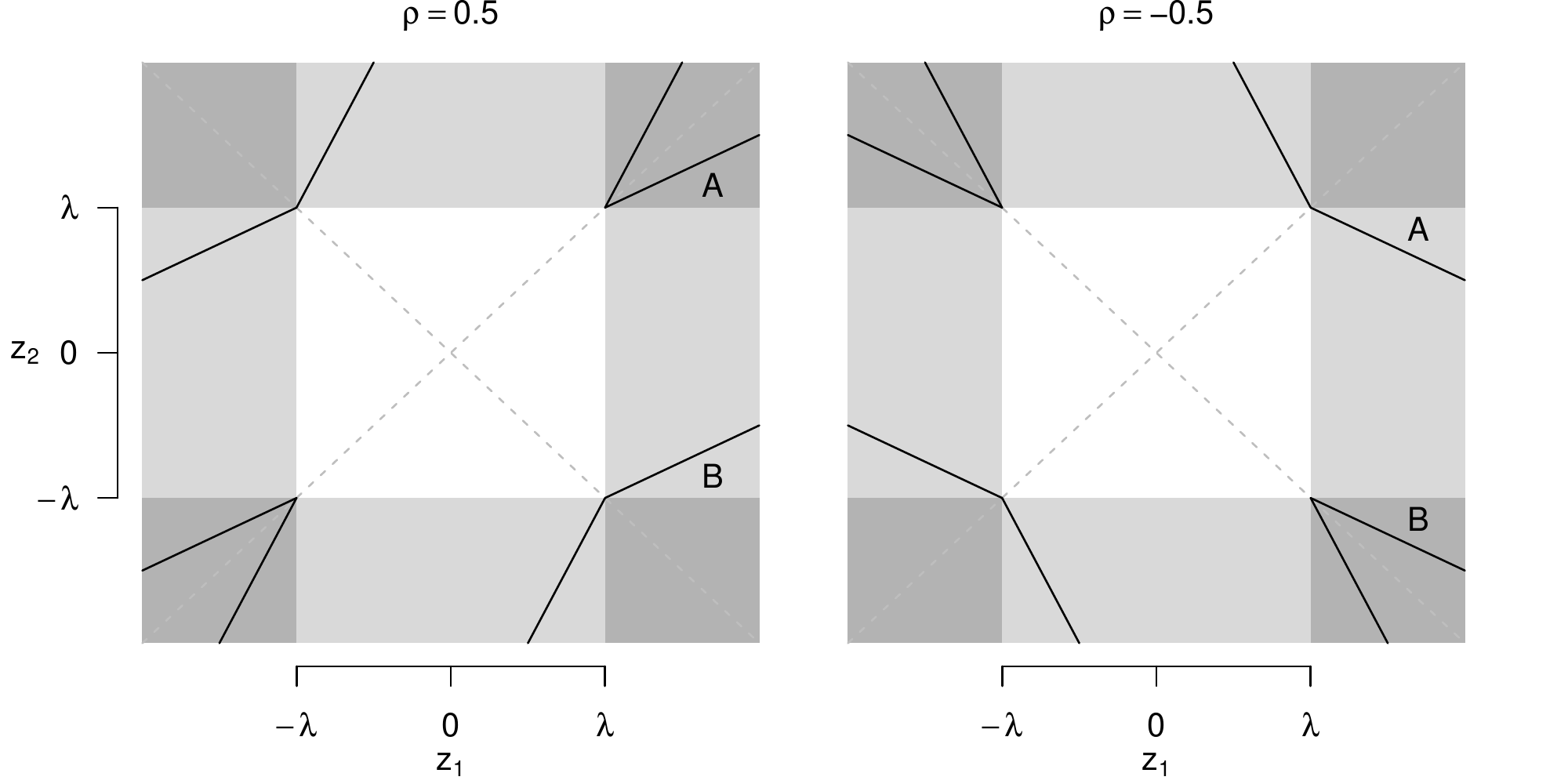}
  \caption{\label{Fig:conjecture} Illustration of how correlation affects selection in the bivariate case.  Features $\x_1$ and $\x_2$ are positively correlated ($\rho=0.5$) on the left and negatively correlated ($\rho=-0.5$) on the right.}
\end{figure}

In the case where $\x_1$ and $\x_2$ are non-orthogonal, the white region (where no features are selected) remains the same, but the conditions under which both features are selected (the light and dark gray regions) change.
By working through the KKT conditions for the bivariate case under various conditions, we can determine these selection boundaries.
For example, the boundary for $\bh_2 > 0$ given that $\bh_1 > 0$ (i.e., given that $z_1 > \lam$) is
\as{\tfrac{1}{n}\abs{\x_2'\r_2} > \lam
  &\implies z_2 - \rho\bh_1 > \lam\\
  &\implies z_2 > \rho z_1 + \lam(1-\rho).}
These boundaries are drawn in black on Figure~\ref{Fig:conjecture}; the preceding equation is the line just above the letter ``A''.
As the figure illustrates, when $\rho > 0$ the region over which $\beta_1, \beta_2 \neq 0$ narrows in the upper right and lower left quadrants, and widens in the other two quadrants.
Considering the difference $\Ex\abs{\cS \cap \cN} - \{\Pr(\abs{z_1}>\lam) + \Pr(\abs{z_2}>\lam)\}$, we see that in many regions, the terms cancel out by integrating the same quantity over the same distribution.
The differences lie in four pairs of triangular regions; one such pair is labeled ``A'' and ``B'' in the figure.
The region A is where two variables are selected under orthogonality, but only one in the correlated case, while region B is where one variable is selected under orthogonality, but two are selected in the correlated case.
However, because $z_1$ and $z_2$ are positively correlated in this scenario, the density at each point in A is higher than its corresponding point in B, and therefore estimator \eqref{eq:fd} is conservative in the sense that it overestimates the expected number of false discoveries.  The opposite scenario happens for $\rho < 0$, where region B now has the higher probability density, which again leads to a larger integral under orthogonality, with equality between the two quantities occurring only at $\rho=0$.

In terms of Theorem~\ref{Thm:main}, this argument implies that when $\Sigma_\cN \neq \I$, the quantity $\tfrac{1}{\sqrt{n}}\x_j'\r_j$ converges to a distribution with thinner tails than $\Norm(0, \sigma^2)$.
Intuitively, this makes sense: if two noise features are correlated, a regression-based method such as the lasso will tend to select a single feature rather than both.
Thus, the uncorrelated case is not just mathematically convenient, it also represents a worst case scenario with respect to the number of noise features that we can expect to be selected by chance alone.

To investigate the robustness of the proposed mFDR estimator in the presence of moderate correlation, let us carry out the following simulation.  The generating model contains 6 independent causative features and 494 correlated noise features, with a 1:1 signal-to-noise ratio ($n=100$, $p=500$, $R^2=0.5$).  The noise features are given an autoregressive correlation structure with $\cor(\x_j,\x_k) = 0.8^{\abs{j-k}}$.  The results of the simulation are shown in Figure~\ref{Fig:sim-corr-auto}.

\begin{figure}[ht!]
  \centering
  \includegraphics[width=0.8\linewidth]{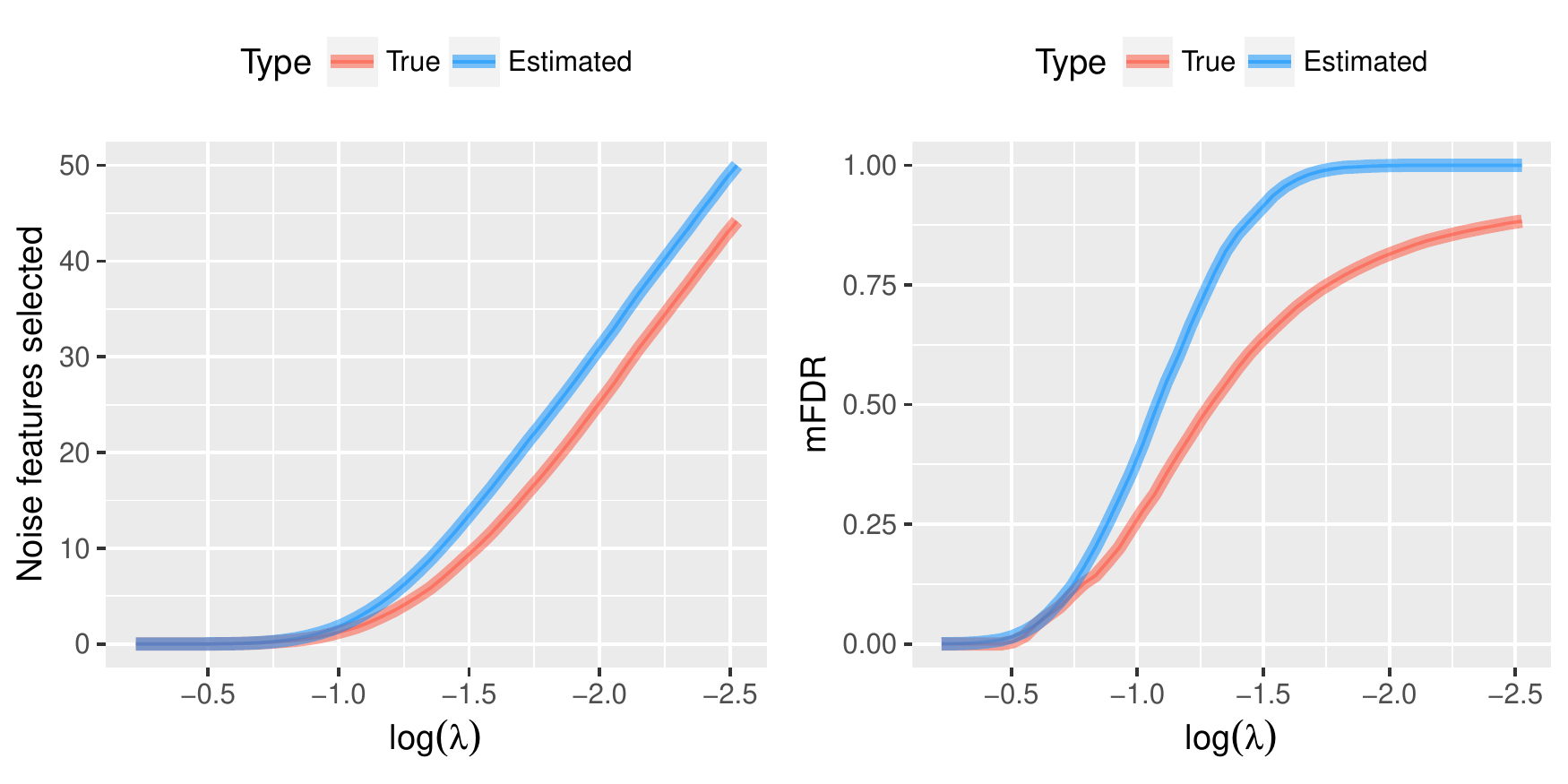}
  \caption{\label{Fig:sim-corr-auto} Accuracy of estimators \eqref{eq:fd} and \eqref{eq:fdr} in the case of (autoregressive) correlated noise features.}
\end{figure}

Compared to Figure~\ref{Fig:sim-ind}, the mFDR estimates are somewhat more conservative in this case, although still quite accurate -- certainly accurate enough to be useful in practice.  For example, at $\lam=0.43$, the true inclusion rate for noise variables was 14\%, while the estimated rate according to \eqref{eq:fd} was 20\%.

This simulation illustrates that although its derivation is based on independent noise features, the proposed mFDR estimator is reasonably robust to the presence of correlation.  Furthermore, to the extent that it is inaccurate, it provides a conservative estimate of the marginal false discovery rate, suggesting (as the conjecture indicates) that the approach provides control over the mFDR, at least on average.

\section{Permutation approach}
\label{Sec:perm}

As correlation between noise features becomes more extreme, the proposed estimator becomes increasingly conservative.  As an extreme example, suppose that the correlation structure from Section~\ref{Sec:cor} was exchangeable instead of autoregressive: $\cor(\x_j,\x_k) = 0.8$ for all $j$, $k$.  The results of this modification to the simulation (all other aspects remaining the same) are shown in Figure~\ref{Fig:sim-corr-ex}.

As one might expect, the estimates are far more conservative in this case.  For example, at $\lam=0.43$, the estimated mFDR is 23\% even though the true mFDR is only 1\%.  This is a consequence of the independence approximation, where \eqref{eq:fdr} operates under the simplifying assumption that the selection of one noise feature does not affect the probability of other noise features being selected.  This is reasonably accurate in most cases, but not in this situation, where all noise features are highly correlated with each other.  As a result, the lasso tends to select only a single noise feature from this highly correlated set, while the independence-based estimate \eqref{eq:fd} indicates that it has likely selected, say, 7 or 8.

This phenomenon is not unique to penalized regression; substantial correlation among features causes problems with conventional false discovery rates as well \citep{Efron2007b}.  One widely used approach for controlling error rates while preserving correlation structures is to use a permutation approach \citep{Westfall1993}, and a similar strategy may be applied in order to calculate marginal false discovery rates for penalized regression as well.  The primary advantage of this approach is that it eliminates the conservatism of the analytic approach developed in Section~\ref{Sec:fir}, while the primary disadvantage is a greatly increased computational burden.  In this section, I describe two permutation-based methods and apply them to the simulation shown in Figure~\ref{Fig:sim-corr-ex}.

\begin{figure}[ht!]
  \centering
  \includegraphics[width=0.5\linewidth]{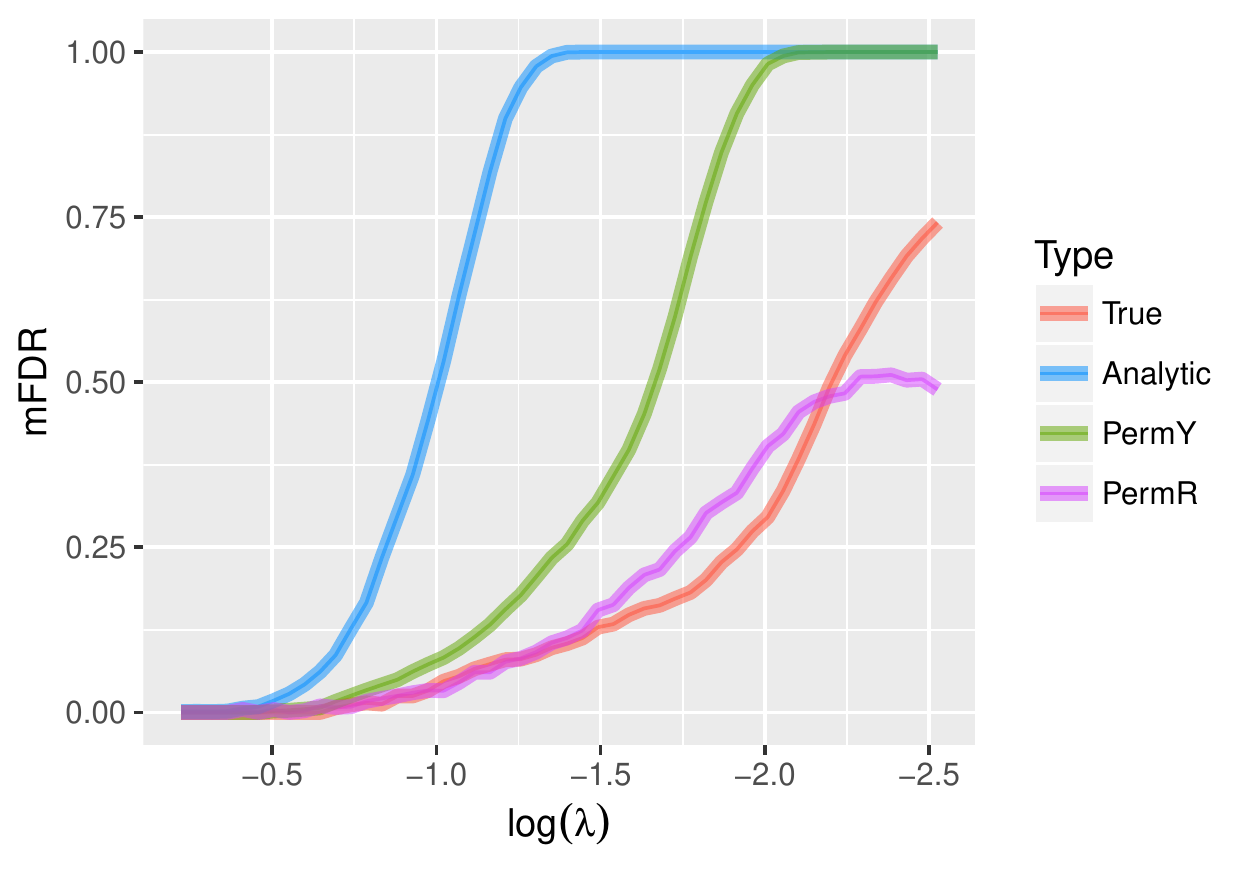}
  \caption{\label{Fig:sim-corr-ex} Accuracy of the analytic \eqref{eq:fdr}, permute-the-outcome (``PermY'', \eqref{eq:permy}), and permute-the-residuals (``PermR'') estimators in the case of (exchangeable) correlated noise features.}
\end{figure}

\subsection{Permuting the outcome}
\label{Sec:permy}

The simplest approach is simply to randomly permute the outcome $\y$, creating new outcomes $\yy^{(b)}$ for $b=1,2,\ldots,B$.  Then, for each permutation $b$, solve for the lasso path $\tbb^{(b)}(\lam;\X,\yy^{(b)})$, estimate the average number of noise features included in the model for a given value of $\lam$
\as{\widehat{\FD}(\lam) = \frac{\sum_b \#\{\tb_j^{(b)}(\lam) \neq 0\}}{B},}
and the marginal false discovery rate using
\al{eq:permy}{\widehat{\mFDR}_{Y}(\lam) = \frac{\widehat{\FD}(\lam)}{S(\lam)},}
where $S(\lam)$ is the number of variables selected by the lasso using the original (i.e., not permuted) data.

This method is applied to the case of heavy exchangeable correlation in Figure~\ref{Fig:sim-corr-ex}.  The method is considerably less conservative than the analytic approach, although still conservative for reasons that will be discussed in Section~\ref{Sec:permr}.  For example, at $\lam=0.43$, the average estimated mFDR is 4\%, where the true value was 1\% and the analytic approach yielded 23\%.

By permuting $Y$, this approach constructs realizations of the data in which all features belong to $\cN$, the set of noise features.  Like \eqref{eq:fdr}, it limits the number of noise features selected but cannot be used to control the number of false discoveries in the sense of limiting selections from all variables with $\beta_j=0$.  In the hypothesis testing literature, this is referred to as ``weak control'' over the error rate.

\subsection{Permuting the residuals}
\label{Sec:permr}

As seen in Figure~\ref{Fig:sim-corr-ex}, permuting the outcomes is still conservative in its estimation of mFDR.  The primary reason for this is that, provided that at least some features are {\em not} noise, the variance of $Y$ can be separated into signal and noise; ideally we would permute only the noise, but by permuting the outcome we permute the signal as well.  This has the effect, in a sense, of overestimating the noise present in the model; this is essentially the same phenomenon that occurs in Theorem~\ref{Thm:main} if the sequence $\sqrt{n}\lam_n$ is not bounded.

One alternative is, rather than permuting the outcome, to permute the residuals, $\r(\lam)=\y - \X\bbh(\lam)$, of the original lasso fit.  The method is otherwise identical to that described in Section~\ref{Sec:permy}, with the notable exception that the residuals, unlike the outcomes, depend on $\lam$ and thus, $B$ separate lasso solutions $\tbb^{(b)}(\lam;\X,\rr^{(b)})$ must be calculated at each value of $\lam$ (in Section~\ref{Sec:permy}, the same solutions could be used for all values of $\lam$), substantially increasing the computational burden.

Nevertheless, this increased computational cost does offer a benefit, as seen in Figure~\ref{Fig:sim-corr-ex}.  Unlike the analytic and permute-the-outcome approaches, permuting the residuals provides an essentially unbiased estimate of the true false inclusion rate except at very small $\lam$.  For example, at $\lam=0.29$, the average estimated mFDR is 8\%, as was the true mFDR, whereas the permute-the-outcome and analytic methods estimated mFDRs of 16\% and 90\%, respectively.

The purpose of this manuscript is primarily to investigate the analytic approach to calculating false inclusion rates outlined in Section~\ref{Sec:fir}, with relatively less emphasis on the permutation approaches.  In the opinion of the author, the analytic approach is almost always useful, since it can be calculated instantly.  Whether one wishes to take the time to carry out the permutation approach, however, depends on context: how correlated the features are, how problematic a somewhat conservative estimate of mFDR is, and how far along in the analytic process one is (initial exploration or final publication results).  For a more detailed comparison of the analytic and permutation approaches in the context of genetic association studies, see \citet{Yi2015}.

\section{Case studies}
\label{Sec:case}

\subsection{Breast cancer gene expression study}

As a case study in applying the proposed method to real data, we will analyze data on gene expression in breast cancer patients from The Cancer Genome Atlas (TCGA) project, available at \url{http://cancergenome.nih.gov}.
In this dataset, expression measurements of 17,814 genes, including BRCA1, from 536 patients are recorded on the log scale.

BRCA1 is a well-studied tumor suppressor gene with a strong relationship to breast cancer risk.
Because BRCA1 is likely to interact with many other genes, including tumor suppressors and regulators of the cell division cycle, it is of interest to find genes with expression levels related to that of BRCA1.
These genes may be functionally related to BRCA1 and are useful candidates for further studies.

\begin{figure}[ht!]
  \centering
  \includegraphics[width=0.6\linewidth]{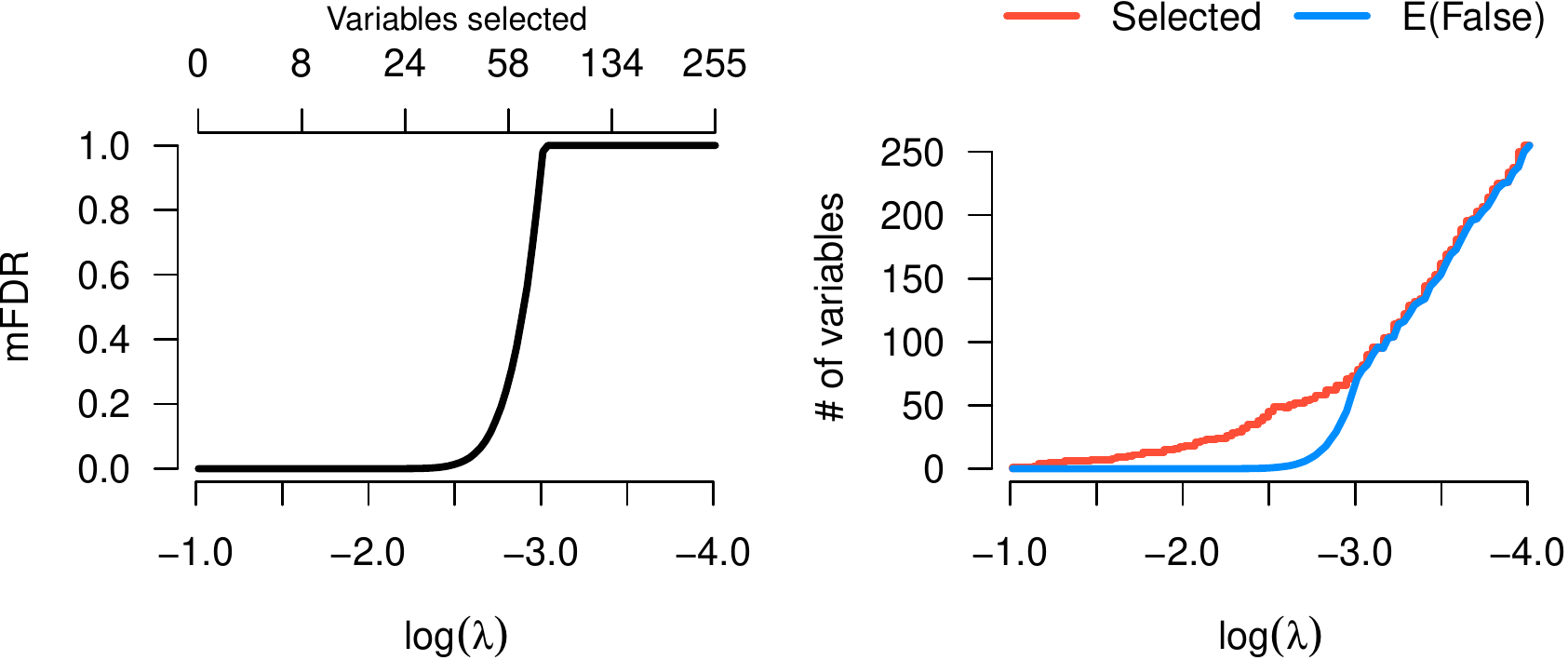}
  \caption{\label{Fig:tcga} False inclusion rate estimates for a lasso model applied to the breast cancer TCGA data.}
\end{figure}

For this analysis, 491 genes with missing data were excluded, resulting in a design matrix with $p=17,322$ predictors.
The resulting mFDR estimates for the lasso solution path are presented in Figure~\ref{Fig:tcga}.
Here, the mFDR estimates indicate that many genes are predictive of BRCA1 expression -- we can safely select 55 variables before the false inclusion rate exceeds 10\%.
This makes sense scientifically, as a large number of genes are known to affect BRCA1 expression through a variety of mechanisms, and the sample size here is sufficient that we should be able to identify many of them.
For the sake of comparison, the univariate approach of separately testing each feature identifies 7,903 genes that have a significant correlation with BRCA1 at a false discovery rate of 10\%.

In contrast, the selective inference and sample-splitting procedures each select just a single feature.
Although the mFDR estimates are slightly conservative as discussed in Section~\ref{Sec:cor}, this matters far less in practice than whether one adopts a marginal or conditional perspective with respect to false discoveries.
This example clearly illustrates how conservative the conditional perspective is in practice, especially with high-dimensional data.

The mFDR approach is also vastly more convenient than sampling splitting or covariance testing from the standpoint of computational burden.  Calculating the mFDR is essentially instantaneous after fitting the lasso path; the entire analysis requires just 1.3 seconds.  Meanwhile, sample splitting required 9 minutes, and applying the pathwise conditional test with the forwardStop rule using the {\tt selectiveInference} package required 1.5 hours.

Interestingly, although sample splitting and pathwise conditional testing each identified a single feature, they did not select the same feature.  Sample splitting selected the gene NBR2, which is unquestionably associated with BRCA1 expression -- the two genes are adjacent to one another on chromosome 17 and share a promoter \citep{Di2010}.  However, NBR2 is the third gene to enter the lasso model.  The forwardStop rule based on pathwise conditional testing recommends stopping after the first variable is added to the model, and thus fails to identify NBR2.

Finally, it is worth comparing these results to the selection of $\lam$ by cross-validation (CV).
For the TCGA data, $\lam=0.0436$ minimizes the CV error.
The estimated mFDR at this value, however, is 77\%, indicating that although this value of $\lam$ may be attractive from a prediction perspective, we cannot be confident that the variables selected by the model are truly related to the outcome.
Indeed, it has long been recognized from a theoretical perspective that while the lasso has attractive variable selection and prediction properties, it cannot achieve both those aims simultaneously \citep{Fan2001}.
The mFDR estimates illustrate this concretely: $\lam=0.0436$ produces accurate predictions, but a larger value, $\lam=0.0627$, is required in order to have confidence that no more than 10\% of the variables selected are false discoveries.

\subsection{Minimax concave penalty}

The mFDR estimator follows directly from the KKT conditions for a given penalty, and it is straightforward to extend to other penalties.
In fact, the KKT conditions for many penalties lead to the same expression \eqref{eq:kkt} and therefore the same mFDR estimator (although as a technical note, for nonconvex optimization problems these conditions are not usually referred to as KKT conditions; they are still necessary but no longer sufficient).

One such penalty is the minimax concave penalty, or MCP \citep{Zhang2010}.
Briefly, the MCP produces sparse estimates like the lasso, but modifies the penalty such that the selected variables are estimated with less shrinkage towards zero (see the original paper for details).
The main consequence of this is that MCP solutions tend to be more sparse than those of the lasso, in that they allow features to have larger regression coefficients, requiring fewer features to achieve the same predictive ability.
To put it differently, for the lasso to estimate large coefficients accurately, it must lower $\lam$, thereby increasing the number of noise features in the model and increasing the mFDR.
MCP, on the other hand, can estimate large coefficients accurately while still screening out the noise features, at least asymptotically (this is known as the ``oracle property'').

To see how this works in practice, we can fit an MCP model to the TCGA data and compare its results to those of the lasso.
For the value of $\lam$ minimizing CV error, both methods had similar predictive accuracy (cross-validated $R^2=0.61$ for the lasso and 0.58 for MCP), but the MCP model used far fewer features (38 compared to 96 for the lasso).
Consequently, for these values of $\lam$ the estimated mFDR for MCP was much lower than that of the lasso: 5\% compared to the lasso's 77\%.
If we restrict the lasso mFDR to 5\% by increasing $\lam$, its predictive accuracy falls to $R^2=0.57$.

This is representative of the relationship between lasso and MCP: the lasso can typically achieve slightly better prediction accuracy, but in doing so selects a large number of noise features.
This pattern has been observed in simulation studies \citep[e.g.,][]{Breheny2011}, but mFDR estimates offer a way to observe and assess this tradeoff in the analysis of real data.

This example also illustrates the ease with which the proposed estimator can be extended to other penalties.
In addition to MCP, the SCAD \citep{Fan2001} and elastic net \citep{Zou2005} also have similar KKT conditions leading to \eqref{eq:kkt} and thus require only trivial modifications to the mFDR estimator.

\subsection{Genetic association study of cardiac fibrosis}
\label{Sec:magnet}

As an additional example, let us also look at a genetic association study of cardiac fibrosis.  The data come from the Myocardial Applied Genomics Network (MAGNet), which collected tissue and gene expression data on 313 human hearts along with genetic data using the Affymetrix Genome-Wide Human SNP Array 6.0.  We are interested in the the ratio of cardiomyocytes to fibroblasts in the heart tissue.
An abundance of fibroblasts is indicative of cardiac fibrosis, which leads to heart failure.
In this analysis, the goal is to discover single nucleotide polymorphisms (SNPs) associated with increased fibrosis.
Here, we use the log of the cardiomyoctye:fibroblast ratio as the response, and SNPs with $<10\%$ missing data as the predictors, resulting in $p=660,496$ features.
{\tt mimimac2} was used to impute the missing genotypes that remained \citep{Fuchsberger2015}.
Neither sample splitting or covariance testing was attempted here due to the size of the data set.

In contrast to the TCGA gata, no features can be selected with any degree of confidence in the MAGNet data.
For all values of $\lam$ in which features are selected, the mFDR estimate is 100\%.
This is consistent with a traditional genome-wide association analysis, which also fails to identify any SNPs that are significant following a Bonferroni correction for multiple testing.

This negative example illustrates the usefulness of mFDR estimates in terms of guarding against false positives.
The efficiency with which methods like the lasso extend to high dimensions make analyses like this (with $n=313$ and $p=660,496$) feasible from a computational standpoint.
However, there are important statistical considerations here that need to be accounted for; namely, with $p$ so large, the probability of selecting noise features is so great that one cannot place any trust in the variables that the lasso selects.
These considerations are clearly illustrated by the estimation of mFDR, but are lost if one simply uses the lasso to identify the top $\abs{\cS}$ SNPs \citep[as in, e.g.,][]{Wu2009}.

\section{Discussion}

The false inclusion rate estimator presented in this article is a straightforward, useful way of quantifying the reliability of feature selection for penalized regression models such as the lasso.
Compared with other proposals such as sample splitting and the pathwise covariance and selective inference tests, mFDR uses a more relaxed definition of a false discovery, in which only variables marginally independent of the outcome are considered to be false discoveries.
The relaxed definition pursued here has several advantages in practice, particularly in high dimensions.
These advantages include greater power to detect features, no added computation cost, and a simple rate with a straightforward interpretation.
Although the estimate can be conservative for highly correlated features, in most realistic settings this conservatism is mild.

The proposed estimator is easy to implement and is available (along with the permutation methods of Section~\ref{Sec:perm}) in the {\tt R} package {\tt ncvreg} \citep{Breheny2011}, which was used to fit all of the models presented in this paper (except where the {\tt hdi} and {\tt selectiveInference} packages were used).
The following bit of code demonstrates its use:
\begin{code}
  fit <- ncvreg(X, y, penalty="lasso")
  obj <- fir(fit)
  plot(obj)
\end{code}
\noindent This will fit a lasso model, estimate the mFDR at each value of $\lam$ and assign it to an object, then plot the results, as in Figure~\ref{Fig:tcga}.

The simplicity of the method makes it available at no added computational cost and very easy to generalize to new methods: the {\tt fir} function in {\tt ncvreg} also works for elastic net, SCAD, MCP, and Mnet \citep{Huang2016} penalties.
As with any statistic, there are limitations and one needs to be careful with interpretation, but mFDR is a convenient summary measure that is helpful in the selection of $\lam$ for penalized regression models as well as a useful estimate for the reliability of feature selection in such models.

\section*{Acknowledgments}

I would like to thank Jian Huang for fruitful discussions of this idea when it was still in its early stages, Ina Hoeschele for discussions of the permutation approach, and Ryan Boudreau for the MAGNet data of Section~\ref{Sec:magnet}.

\section*{Appendix}

All of the code and data to reproduce the analyses in this manuscript, with the exception of the human genetic data considered in Section~\ref{Sec:magnet}, are available at \url{https://github.com/pbreheny/lassoFDRpaper}.

\begin{proof}[Proof of Theorem~\ref{Thm:main}]
  Starting with the observation that $\r_j = \X\bb + \bvep - \X_{-j}\bbhj$, for any $j \in \cN$ we have
  \as{\frac{1}{\sqrt{n}}\x_j'\r_j &= \frac{1}{\sqrt{n}}\x_j'\bvep +
    (\tfrac{1}{n}\x_j'\X_{-j})\{\sqrt{n}(\bbj - \bbhj)\}}
  The first term, $\frac{1}{\sqrt{n}}\x_j'\bvep$, follows a $N(0, \sigma^2)$ distribution by the independence of $\x_j$ and $\bvep$.  The second term, $\tfrac{1}{n}\x_j'\X_{-j}$, converges to zero by the assumption that $\Sigma_{\cN}=\I$ ($\x_j$ is also independent of variables in $\cA$ since $j \in \cN$).  Finally, the third term is bounded in probability provided that $\sqrt{n}\lam_n$ is bounded \citep{Fan2001}.  Therefore, by Slutsky's Theorem, the entire quantity converges in distribution to $N(0, \sigma^2)$.
\end{proof}

\bibliographystyle{ims-nourl}

\end{document}